\documentclass{article}

\usepackage{authblk}

\usepackage[utf8]{inputenc}
\usepackage[T1]{fontenc}
\usepackage[english]{babel}
\usepackage{textcomp}
\usepackage{amsmath,amssymb,amsthm}
\usepackage{lmodern}
\usepackage[a4paper]{geometry}
\usepackage{xcolor, pict2e}
\usepackage{microtype}
\usepackage{caption}
\usepackage{listings}
\usepackage{multicol}
\usepackage{moreverb}
\usepackage{hyperref}
\hypersetup{pdfstartview=XYZ}
\usepackage{wrapfig}
\usepackage[sans]{dsfont}

\usepackage{tikz, pgfplots}
\pgfplotsset{compat=1.18}
\usetikzlibrary{positioning}
\usepackage{amsmath,amssymb}
\usetikzlibrary{decorations.pathreplacing}

\usepackage{ytableau}

\usepackage{hyperref}
\hypersetup{
    colorlinks=true,
    linkcolor=blue,
    citecolor=magenta,
    urlcolor=blue,
    pdfborder={0 0 0}
}

\newcommand{\kS}{{\ensuremath{\mathfrak{S}}} }
\newcommand{\kA}{{\ensuremath{\mathfrak{A}}} }

\definecolor{fond}{rgb}{0.05,0.05,0.25}

\DeclareMathOperator{\Poiss}{Poiss}

\DeclareMathOperator{\dtv}{d_{TV}}
\DeclareMathOperator{\dsep}{d_{sep}}

\DeclareMathOperator{\mlh}{mlh}

\DeclareMathOperator{\ch}{\mathrm{ch}}
\DeclareMathOperator{\Id}{\mathrm{Id}}

\DeclareMathOperator{\CTST}{\mathcal{CTST}}
\DeclareMathOperator{\FIT}{FIT}
\DeclareMathOperator{\FITStates}{FITStates}

\usepackage{todonotes}

\usepackage{amsmath,amssymb,amsthm}
\usepackage{mathtools,mathrsfs}
\usepackage{dsfont}

\usepackage[font=sf, labelfont={sf,bf}, margin=1cm]{caption}
\usepackage{enumitem}
\setlist[itemize,1]{nosep}
\setlist[enumerate,1]{nosep,label=(\alph*)}

\usepackage{float}
\usepackage[caption = false]{subfig}
\usepackage{graphicx}

\newtheorem*{theorem*}{Theorem}
\newtheorem{theorem}{Theorem}[section]
\newtheorem{proposition}[theorem]{Proposition}
\newtheorem{lemma}[theorem]{Lemma}

\newtheorem{definition}[theorem]{Definition}
\newtheorem{conjecture}[theorem]{Conjecture}
\newtheorem{question}[theorem]{Question}
\theoremstyle{remark}
\newtheorem{remark}[theorem]{Remark}

\newcommand{\cyan}[1]{\textcolor{cyan}{ #1 }}

\newcommand{\cA}{{\ensuremath{\mathcal A}} }
\newcommand{\cB}{{\ensuremath{\mathcal B}} }
\newcommand{\cC}{{\ensuremath{\mathcal C}} }
\newcommand{\cD}{{\ensuremath{\mathcal D}} }
\newcommand{\cE}{{\ensuremath{\mathcal E}} }

\newcommand{\cL}{{\ensuremath{\mathcal L}} }

\newcommand{\cN}{{\ensuremath{\mathcal N}} }

\newcommand{\cP}{{\ensuremath{\mathcal P}} }

\newcommand{\bbE}{{\ensuremath{\mathbb E}} }
\newcommand{\bbN}{{\ensuremath{\mathbb N}} }

\newcommand{\bbR}{{\ensuremath{\mathbb R}} }

\newcommand{\bbP}{{\ensuremath{\mathbb P}} }

\newcommand{\ag}{\left\{ } 

\newcommand{\ad}{\right\} }

\newcommand{\cg}{\left[}
\newcommand{\cd}{\right]}
\newcommand{\pg}{\left(} 
\newcommand{\pd}{\right)}
\newcommand{\bg}{\left|}
\newcommand{\bd}{\right|}
\newcommand{\lf}{\left\lfloor}
\newcommand{\rf}{\right\rfloor}
\newcommand{\lc}{\left\lceil}
\newcommand{\rc}{\right\rceil}

\newcommand{\du}{{\ensuremath{\;:\;}}} 

\makeatletter
\newcommand*\bigcdot{\mathpalette\bigcdot@{.5}}
\newcommand*\bigcdot@[2]{\mathbin{\vcenter{\hbox{\scalebox{#2}{$\m@th#1\bullet$}}}}}
\makeatother

\numberwithin{equation}{section}

\title{Every cutoff profile is possible
}
\author{Lucas Teyssier}
\affil{University of British Columbia, \texttt{teyssier@math.ubc.ca}}
\date{\today}

\begin{document}

\renewcommand{\theparagraph}{\thesubsection.\arabic{paragraph}} 

\maketitle

\begin{abstract}
We introduce fruit-inosculated-tree Markov chains. These chains have easily tunable parameters and are a good source of examples. In particular, we prove that every cutoff profile is possible, with any cutoff time and window size. 
\end{abstract}

\renewcommand\abstractname{Résumé}
\begin{abstract}
Nous introduisons les arbres à inosculation en un fruit. Ces chaines de Markov ont des paramètres facilement ajustables et sont une bonne source d’exemples. Nous montrons en particulier que tous les profils de coupure sont possibles, pour n’importe quel temps de coupure et taille de fenêtre.
\end{abstract}

\renewcommand\abstractname{Resumo}
\begin{abstract}
Ni enkondukas arbojn kun inoskuliĝo je frukto. Tiuj Markov-ĉenoj havas facile agordeblajn parametrojn kaj estas bona fonto de ekzemploj. Ni montras ke ĉiuj tranĉoprofiloj eblas, por ajna tranĉotempo kaj fenestra grandeco.
\end{abstract}

\tableofcontents

\section{Introduction}

\subsection{Cutoff phenomenon}

Let $(X_t)_{t\geq 0}$ be an ergodic Markov chain on a finite state space $S$, and denote its stationary measure by $\pi$. For $t\geq 0$ and $x\in S$, denote the law of $X_t$ conditioned on $X_0 = x$  by $\mu_{x,t}$. Define the (worst case, total variation) distance to stationarity of the chain by
\begin{equation}
    \mathrm{d}(t) = \max_{x\in S}\dtv\pg \mu_{x,t}, \pi \pd, 
\end{equation}
where the total variation distance between two probability measures $\mu$ and $\nu$ on $S$ is given by
\begin{equation}
    \dtv(\mu,\nu) = \max_{A \subset S} \bg \mu(A) - \nu(A)\bd  = \frac{1}{2} \sum_{x\in S}\bg \mu(x) - \nu(x)\bd.
\end{equation}
\begin{definition}
    Let $((X^{(n)}_t)_{t\geq 0})_{n\geq 1}$ be a sequence of ergodic Markov chains, in discrete or continuous time. We say that that $((X^{(n)}_t)_{t\geq 0})_{n\geq 1}$ exhibits a total variation \textbf{cutoff} at a (sequence of) times $\pg t_n \pd_{n \geq 1}$ if for every $\varepsilon>0$, we have 
    \begin{equation}
        \mathrm{d}^{(n)}\pg \lf (1-\varepsilon) t_n \rf \pd \xrightarrow[n\to \infty]{} 1 \quad \text{ and } \quad \mathrm{d}^{(n)}\pg \lc (1+\varepsilon) t_n \rc \pd \xrightarrow[n\to \infty]{} 0,
    \end{equation}
where for $n\geq 1$, $\mathrm{d}^{(n)}(\cdot)$ is the distance to stationarity for the chain $(X^{(n)}_t)_{t\geq 0}$.
\end{definition}

The cutoff phenomenon was proved for random transpositions by Diaconis and Shahshahani \cite{DiaconisShahshahani1981}, and for the riffle shuffle by Aldous \cite{Aldous1983mixing}. It is now known for hundreds of models and in a variety of contexts, including models of statistical physics \cite{LubetzkySly2013, Lacoin2016cutoffadjacent, LabbéLacoin2019}, diffusions on compact Lie groups \cite{Méliot2014}, random walks on Ramanujan graphs \cite{LubetzkyPeres2016}, and for processes with non-negative curvature \cite{Salez2024, PedrottiSalez2025newcriterion} under mild additional conditions. We refer to the books \cite{LivreLevinPeres2019MarkovChainsAndMixingTimesSecondEdition, LivreDiaconisFulman2023} for more background.

\medskip

\subsection{Cutoff profiles}\label{s: cutoff profiles definitions and literature}

For some models it is possible to push the analysis further, and to also understand how the distance to stationarity behaves within the phase transition. If $(a_n)$ and $(b_n)$ are sequences of positive real numbers, we write $a_n=o(b_n)$ or $a_n \lll b_n$ if $a_n/b_n \to 0$ as $n\to \infty$. 

\begin{definition}
    A \textbf{continuous time scaling triplet} is a triplet $((t_n), (w_n),p)$, where $(t_n)$ and $(w_n)$ are two sequences of (positive) real numbers such that $\sqrt{t_n}\lll w_n \lll t_n$ as $n\to \infty$, and $p\du \bbR \to [0,1]$ is a weakly decreasing continuous function such that $p(c) \xrightarrow[c\to - \infty]{} 1$ and  $p(c) \xrightarrow[c\to + \infty]{} 0$. We denote the set of all continuous time scaling triplets by $\CTST$.
\end{definition}

\begin{definition}
Let $((X^{(n)}_t)_{t\geq 0})_{n\geq 1}$ be a sequence of continuous-time Markov chains. If there exists $((t_n), (w_n),p) \in \CTST$ such that for every $c\in \bbR$ we have
\begin{equation}
    \mathrm{d}^{(n)}\pg t_n + cw_n \pd \xrightarrow[n\to \infty]{} p(c),
\end{equation}
then the function $p$ is called \textbf{cutoff profile} and $w_n$ is called  \textbf{window size}.
\end{definition}

\begin{remark}\label{rem: equivalence scaling triplets}
   To be precise, cutoff profiles are defined uniquely only up to equivalence of scaling triplets, where two scaling triplets $((t_{1,n}, (w_{1,n}),p_1)$ and $((t_{2,n}), (w_{2,n}),p_2)$ are \textbf{equivalent} if there exist $\alpha, \beta \in \bbR_+^*$ such that 
\begin{enumerate}
    \item $p_2(\alpha c + \beta) = p_1(c)$ for every $c\in \bbR$, 
    \item $w_{1,n} = (\alpha + o(1))w_{2,n}$,
    \item $t_{1,n} = t_{2,n} + (\beta + o(1)) w_{2,n}$ as $n\to \infty$.
\end{enumerate}
For most classical Markov chains, there is a natural choice of scaling. It is then standard to omit the equivalence relation and call the function $p$ of such a scaling triplet \textit{the} cutoff profile. Similarly, what matters for the window size is its order of magnitude, rather than the precise constant in front.
\end{remark}

\begin{remark}
    We also note that the sequences $(t_n)$ and $(w_n)$ may not have asymptotic equivalents: for the affine walks on the hypercube studied in \cite{DiaconisGraham1992AffineWalkHypercube} both sequences have an oscillating behaviour!
\end{remark}

Historically, cutoff profiles were found only for a handful of Markov chains: the simple random walk on the hypercube \cite{DiaconisGrahamMorrison1990hypercube}, the top-to-random shuffle \cite{DiaconisFillPitman1992toptorandom}, and the riffle shuffle \cite{BayerDiaconis1992}. 

Cutoff profiles can have different shapes. The most common one is the Gaussian shape $c\mapsto \dtv\pg \cN(0,1), \cN(e^{-c}, 1)  \pd$ of the hypercube and riffle shuffle \cite{DiaconisGrahamMorrison1990hypercube, BayerDiaconis1992}, and the second most common one is the Poissonian shape $c\mapsto \dtv\pg \Poiss(1), \Poiss(e^{-c}) \pd$ of random transpositions \cite{Teyssier2020}.
Less common profiles include the free Meixner shape of diffusions on quantum groups \cite{FreslonTeyssierWang2022}, and shapes related to the largest eigenvalue of random matrices as for the asymmetric simple exclusion process on a segment \cite{BufetovNejjar2022}. 

\medskip

In the last years a variety of new profiles has been found, and techniques to find cutoff profiles are blooming. Below, we attempt to provide a comprehensive list of known profiles to date. 

\medskip

Apart from the hypercube and the riffle shuffle, a Gaussian profile was proved for simple random walks on Ramanujan graphs \cite{LubetzkyPeres2016}, the simple exclusion process on the segment \cite{Lacoin2016cutoffprofileexclusioncircle}, for the averaging process \cite{ChatterjeeDiaconisSlyZhang2022repeatedaverages}, for generalized Ehrenfest urns \cite{NestoridiOlesker-Taylor2022limitprofiles} (whose arXiv version also contains a short proof for the hypercube), the Bernoulli--Laplace urn model \cite{Olesker-TaylorSchmid2025} with many particles, and the shelf shuffle \cite{ChenOttolini2025shelfshuffle}.
Normal distributions appear as limits of various types of observables in probability. For Gaussian profiles there is often a hidden central limit theorem, and it is therefore not a surprise that this shape arises in different contexts and is the most common one.

\medskip
 
The Poissonian shape typically appears for random walks on symmetric groups, when the last observable to mix is the number of fixed points. Note that $\Poiss(1)$ is the distribution of the number of fixed points of a uniformly random permutation, asymptotically.
After being found for random transpositions, the Poissonian profile was extended to cycles of length $o(n)$ in \cite{NestoridiOlesker-Taylor2022limitprofiles}, is now known for conjugacy invariant walks up to a large support \cite{Teyssier2025ProfilesConjugacyInvariant}. Nestoridi \cite{Nestoridi2024comparisonstar} developed a comparison technique that enabled extending this profile to chains that are not conjugacy invariant, such as star transpositions; and this technique was later used to find the cutoff profile of other related models \cite{GhoshKumari2025profilewarptoptwo, ArfaeeNestoridi2025JucysMurphyTranspositions}. A Poissonian profile with adjusted parameters is also known for some particle systems that can be obtained as projections of random transpositions \cite{NestoridiOlesker-Taylor2024ProfilesProjections}. Finally, while the proofs of the aforementioned Poissonian profiles rely on the representation theory of symmetric groups to understand compensations between the eigenvectors of these chains, a more probabilistic proof of the profile for random transpositions was recently found \cite{JainSawhney2024transpositionprofileotherproof}: it still crucially relies on representation theory for the eigenvalues, but replaces the understanding of the eigenvectors by a probabilistic extraction of the number of fixed points.

\medskip

Some recently found cutoff profiles have \textit{Tracy--Widom} shapes, which are the cumulative distribution function of the largest eigenvalue of different models of random matrices: a GUE shape for the asymmetric simple exclusion process (ASEP) on a segment \cite{BufetovNejjar2022}, a GOE shape for the Metropolis biased card shuffling \cite{Zhang2024CutoffProfileMetropolisBiasedCardShuffling}, and shapes interpolating between GOE and GSE for the ASEP with one open boundary \cite{HeSchmid2023ProfileASEPOneOpenBoundary}. The ASEP can also exhibit limit profiles given by the one-point marginals of Airy processes \cite{BernalNejjar2025LimitProfilesASEP}.

\medskip

Méliot \cite{Méliot2014} proved the cutoff for the diffusion on compact Lie groups and computed explicit formulas for densities at any time. However there was so far no success in computing the associated cutoff profiles. The cutoff profiles were effectively computed on some related quantum objects such as the quantum orthogonal group and for quantum random transpositions \cite{FreslonTeyssierWang2022}. Recently, Delhaye \cite{Delhaye2024unitary} defined a Brownian motion on the quantum unitary group, proved that it exhibits a cutoff, and found the $c>0$ part of the cutoff profile.

\medskip

The contributions of the largest eigenvalues taken individually are typically small in the $c>0$ part of the profile, while for $c<0$, they are all large and compensate. This causes the speed of convergence to the profile to be much faster on the right part of the window, and explains why in some cases it is much harder to understand the $c<0$ part than the $c>0$ one. This issue leads to further complications in the quantum setting: in \cite{FreslonTeyssierWang2022, Delhaye2024unitary}, the distribution of the processes are absolutely continuous with respect to the Haar state in the $c>0$ part of the profile, but not in the $c<0$ part. Going around the lack of absolute continuity requires extracting atoms, and the formulas of the cutoff profiles involve distributions that have atoms in their $c<0$ part.

\medskip

Let us also mention that for any (sequence of) discrete time Markov chain exhibiting cutoff, adding enough laziness leads to a Gaussian profile. Of course, obtaining a profile in such a way would not improve the understanding of the chain and would therefore not be interesting.

Doing the opposite can however be very interesting: in some cases there is some structure and laziness prevents us from seeing it. In these cases removing the laziness from the chain can lead to a more precise analysis of the important parameters. To give a concrete example, for conjugacy invariant walks on $\mathfrak{S}_n$ with laziness $1/2$ (or when considered in continuous time), the profile is inherently Poissonian but becomes Gaussian if “$w_n = o(\sqrt{t_n})$”, that is for conjugacy classes with support sizes diverging faster than $n/\ln n$. 
\subsection{Main results}

As described in Section \ref{s: cutoff profiles definitions and literature}, several shapes of cutoff profiles were found. It is natural to wonder if any shape is possible. We prove that this is the case. 

\begin{theorem}\label{thm: main theorem in simple continuous time form}
For every continuous time scaling triplet $((t_n), (w_n),p)$, there exists a sequence of continuous time Markov chains $((X^{(n)}_t)_{t\in \bbR_+})_{n\geq 1}$ such that for every $c\in \bbR$, we have 
\begin{equation*}
\mathrm{d}^{(n)}\pg t_n + c w_n\pd \xrightarrow[n\to+\infty]{} p(c).
\end{equation*}    
\end{theorem}

We may relax the continuity assumption in the definition of scaling triplets for existence. However, if it is relaxed, then the relation described in Remark \ref{rem: equivalence scaling triplets} is no longer an equivalence relation and sequences of Markov chains might exhibit non-trivial behaviours in different \textit{windows}. This contrasts with a recent result of Nestoridi \cite{Nestoridi2025continuityprofiles}, proving the existence of a unique scale and that the cutoff profile must be continuous, under some spectral conditions. We notably give an example of a sequence of Markov chains having a non-trivial behaviour in uncountably many \textit{windows} in Section \ref{s: A non-trivial behaviour in uncountably many windows}, and an example of a sequence of Markov chains having countably many \textit{nested windows} in Section \ref{s: Countably many nested windows}.

\medskip

We write $\bbN = \ag 0, 1, \ldots\ad$ and $\bbN^* = \bbN\backslash \ag 0 \ad = \ag 1,2, \ldots \ad$. For discrete time Markov chains we prove the following.

\begin{theorem} \label{thm: discrete time profiles intro}
    Let $(t_n)$ be a sequence of positive integers and $(w_n)$ be a sequence of positive real numbers such that $w_n = o(t_n)$ as $n\to \infty$. Let $p\du \bbR \to [0,1]$ be a weakly decreasing function (but not necessarily continuous) such that $p(x) \xrightarrow[x\to -\infty]{} 1$ and  $p(x) \xrightarrow[x\to +\infty]{} 0$. There exists a sequence of discrete time Markov chains $((X_t^{(n)})_{t\in \bbN})_{n\geq 1}$ such that
\begin{equation}
    \sup_{c\in \bbR \du t_n + cw_n \in \bbN} \bg \mathrm{d}^{(n)}(t_n + c w_n) - p(c) \bd \xrightarrow[n\to \infty]{} 0.
\end{equation}
\end{theorem}
   
\section{A tree whose branches merge at a single fruit}

In this section, we define a family of Markov chains whose parameters can be easily modified to exhibit various behaviours. The main idea is to define almost deterministic Markov chains by adding a strong drift, and to choose the length of the \textit{branches} so that the worst starting vertex is always the same, and so that the chain reaches the \textit{center of mass} with the desired probability at the desired time. These chains are defined below in Definition \ref{def: FIT Markov chains} and illustrated on Figure \ref{fig:FITexample}. They can be visualized as trees whose branches all join at a single fruit. We call these Markov chains fruit-inosculated-tree (FIT) Markov chains.\footnote{In botany, inosculation is the phenomenon of branches merging together.}

FIT Markov chains can be seen as a generalization of an example constructed by Aldous during the ARCC workshop “Sharp thresholds for mixing times” in 2004. The original document is available at \url{https://www.aimath.org/WWN/mixingtimes/mixingtimes.pdf}, and Aldous' example is described in \cite[Figure 18.2]{LivreLevinPeres2019MarkovChainsAndMixingTimesSecondEdition}. We thank Justin Salez for bringing this example to our attention. Other generalizations of Aldous' example can be found in \cite{HermonLacoinPeres2016, HermonPeres2018SensitivityMixingTimesCutoff}.

\medskip

First we define a set of parameters as follows. Recall that $\bbN^* = \ag 1, 2, \ldots \ad$.
\begin{definition}
   The set $\cP$ is the set of tuples $(k, \boldsymbol{\ell},\boldsymbol{\rho})$, where $k\geq 2$ is an integer, $\boldsymbol{\ell} = (\ell_0, \ell_1,..., \ell_k)\in (\bbN^*)^{k+1}$ satisfies $\ell_0 = \max_{0\leq i\leq k} \ell_i$ and $1\leq \ell_1<\ell_2<\ldots<\ell_k$, and $\boldsymbol{\rho} = (\rho_1, ..., \rho_k)\in (\bbR_+^*)^k$ satisfies $\sum_{i=1}^k\rho_i = 1$.
\end{definition}
 
Then we define the fruit-inosculated-tree state space.

\begin{definition}\label{def: FITstates}
    Let $(k, \boldsymbol{\ell},\boldsymbol{\rho}) \in \cP$. The fruit-inosculated-tree state space with parameters $(k, \boldsymbol{\ell},\boldsymbol{\rho})$, denoted by $\FITStates(k, \boldsymbol{\ell},\boldsymbol{\rho})$ is the set of the following states: 
\begin{itemize}
    \item $x_{j}$ for $0\leq j \leq \ell_0$,
    \item $y_{j}^i$ for $1\leq i \leq k$ and $\ell_0 + 1 \leq j \leq \ell_0 + \ell_i - 1$,
    \item $z$.
\end{itemize}
\end{definition}

Finally we define fruit-inosculated-tree Markov chains.

\begin{definition}\label{def: FIT Markov chains}
    Let $(k, \boldsymbol{\ell},\boldsymbol{\rho}) \in \cP$ and $\varepsilon\in (0,1/2)$. Set $\eta = 1-\varepsilon$. The fruit-inosculated-tree (FIT) Markov chain with parameters $(k, \boldsymbol{\ell},\boldsymbol{\rho},\varepsilon)$, denoted $\FIT(k, \boldsymbol{\ell},\boldsymbol{\rho},\varepsilon)$, is the Markov chain with state space $\FITStates(k, \boldsymbol{\ell},\boldsymbol{\rho})$ and transition matrix $P$ such that,
\begin{itemize}
    \item $P(x_{j},x_{j+1}) = \eta$ and $P(x_{j+1},x_{j}) = \varepsilon$ for $0\leq j \leq \ell_0-1$,
    \item $\bbP(x_0, x_0) = \varepsilon$ and $\bbP(z,z) = \eta$;
\end{itemize}
for $1\leq i\leq k$ such that $\ell_i \geq 2$:
\begin{itemize}
    \item $P(x_{\ell_0}, y_{\ell_0+1}^i) = \eta\rho_i$ and $P(y_{\ell_0+1}^i,x_{\ell_0}) = \varepsilon$,
    \item $P(y_{j}^i,y_{j+1}^i) = \eta $ and $P(y_{j+1}^i,y_{j}^i) = \varepsilon $ for $\ell_0+1 \leq j \leq \ell_0+\ell_i-2$,
    \item $P(y_{\ell_0 + \ell_i -1}^i, z) = \eta$ and $\bbP(z, y_{\ell_0 + \ell_i-1}^i) = \varepsilon\rho_i$;
\end{itemize}
and if $\ell_1 = 1$:
\begin{itemize}
    \item $P(x_{\ell_0} ,z) = \eta \rho_1$ and $P(z, x_{\ell_0}) = \varepsilon \rho_1$.
\end{itemize}

\end{definition}

  \begin{figure}[!ht]
\begin{center}
    \begin{tikzpicture}[->,>=stealth,shorten >=1pt,auto,node distance=3cm, thick, scale=1.05]
\tikzstyle{state} =[circle, draw, fill=white, inner sep=3pt, minimum width=10pt, scale = 1]

    \node[circle, draw, fill=brown, inner sep=3pt, minimum width=10pt, scale = 1] (0) at (0,0) {{\small $x_0$}};
    \node[circle, draw, fill=brown, inner sep=3pt, minimum width=10pt, scale = 1] (1) at (0,1) {{\small $x_1$}};
    \node[circle, draw, fill=brown, inner sep=3pt, minimum width=10pt, scale = 1] (2) at (0,2) {{\small $x_2$}};
    \node[circle, draw, fill=brown, inner sep=3pt, minimum width=10pt, scale = 1] (3) at (0,3) {{\small $x_3$}};
    \node[circle, draw, fill=brown, inner sep=3pt, minimum width=10pt, scale = 1] (4) at (0,4) {{\small $x_4$}};
    \node[circle, draw, fill=brown, inner sep=3pt, minimum width=10pt, scale = 1] (5) at (0,5) {{\small $x_5$}};
    \node[circle, draw, fill=brown, inner sep=3pt, minimum width=10pt, scale = 1] (6) at (0,6) {{\small $x_6$}};

    \node[circle, draw, fill=green, inner sep=3pt, minimum width=10pt, scale = 1] (11) at (-4.5,9) {{\small $y_{7}^1$}};
    \node[circle, draw, fill=green, inner sep=3pt, minimum width=10pt, scale = 1] (21) at (-1.5,9) {{\small $y_{7}^2$}};
    \node[circle, draw, fill=green, inner sep=3pt, minimum width=10pt, scale = 1] (22) at (-1.5,10) {{\small $y_{8}^2$}};
    \node[circle, draw, fill=green, inner sep=3pt, minimum width=10pt, scale = 1] (31) at (1.5,9) {{\small $y_{7}^3$}};
    \node[circle, draw, fill=green, inner sep=3pt, minimum width=10pt, scale = 1] (32) at (1.5,10) {{\small $y_{8}^3$}};
    \node[circle, draw, fill=green, inner sep=3pt, minimum width=10pt, scale = 1] (33) at (1.5,11) {{\small $y_{9}^3$}};
    \node[circle, draw, fill=green, inner sep=3pt, minimum width=10pt, scale = 1] (41) at (4.5,9) {{\small $y_{7}^4$}};
    \node[circle, draw, fill=green, inner sep=3pt, minimum width=10pt, scale = 1] (42) at (4.5,10) {{\small $y_{8}^4$}};
    \node[circle, draw, fill=green, inner sep=3pt, minimum width=10pt, scale = 1] (43) at (4.5,11) {{\small $y_{9}^4$}};
    \node[circle, draw, fill=green, inner sep=3pt, minimum width=10pt, scale = 1] (44) at (4.5,12) {{\small $y_{10}^4$}};
    \node[circle, draw, fill=green, inner sep=3pt, minimum width=10pt, scale = 1] (45) at (4.5,13) {{\small $y_{11}^4$}};

    \node[circle, draw, fill=red, inner sep=3pt, minimum width=40pt, scale = 1] (100) at (0,14) {$z$};

\draw[->, dashed, draw=gray, loop right] (0) edge node {} (0);

    \foreach \x in {0,1,2,3,4,5,21,31,32,41,42,43,44} \pgfmathtruncatemacro{\nextx}{\x+1}
    \draw[->, draw=blue] (\x) edge node {} (\nextx);

    \foreach \x in {0,1,2,3,4,5,21,31,32,41,42,43,44} \pgfmathtruncatemacro{\nextx}{\x+1}
    \draw[->, dashed, draw=gray, bend right=25] (\nextx) edge node {} (\x) ;

    \draw[->, draw=cyan] (6) edge node {\cyan{$\eta\rho_1$}} (11);
    \draw[->, draw=cyan] (6) edge node {\cyan{$\eta\rho_2$}} (21);
    \draw[->, draw=cyan] (6) edge node {\cyan{$\eta\rho_3$}} (31);
    \draw[->, draw=cyan] (6) edge node {\cyan{$\eta\rho_4$}} (41);

    \draw[->, dashed, draw=gray, bend right=25] (11) edge node {} (6);
    \draw[->, dashed, draw=gray, bend left=25] (21) edge node {} (6);
    \draw[->, dashed, draw=gray, bend left=25] (31) edge node {} (6);
    \draw[->, dashed, draw=gray, bend left=25] (41) edge node {} (6);

    \foreach \x in {11, 22,33,45}
    \draw[->, draw=blue] (\x) edge node {} (100);

    \draw[->, dashed, draw=cyan, bend right=25] (100) edge node {\cyan{$\varepsilon\rho_1$}} (11);
    \draw[->, dashed, draw=cyan, bend left=25] (100) edge node {\cyan{$\varepsilon\rho_2$}} (22);
    \draw[->, dashed, draw=cyan, bend left=25] (100) edge node {\cyan{$\varepsilon\rho_3$}} (33);
    \draw[->, dashed, draw=cyan, bend left=25] (100) edge node {\cyan{$\varepsilon\rho_4$}} (45);
    \draw[->, draw=blue, loop left] (100) edge node {} (100);
\end{tikzpicture}
\end{center}
\caption{Fruit-inosculated-tree chain $\FIT(k, \boldsymbol{\ell},\boldsymbol{\rho},\varepsilon)$ with $k=4$ branches of lengths $\boldsymbol{\ell} = (2, 3, 4, 6)$. $\boldsymbol{\rho}$ and $\varepsilon$ are not specified. To lighten the figure, we represent transition rates $\eta=1-\varepsilon$ by continuous blue arrows, and transition rates $\varepsilon$ by dashed grey arrows. Other rates are represented in cyan.}
    \label{fig:FITexample}
\end{figure}
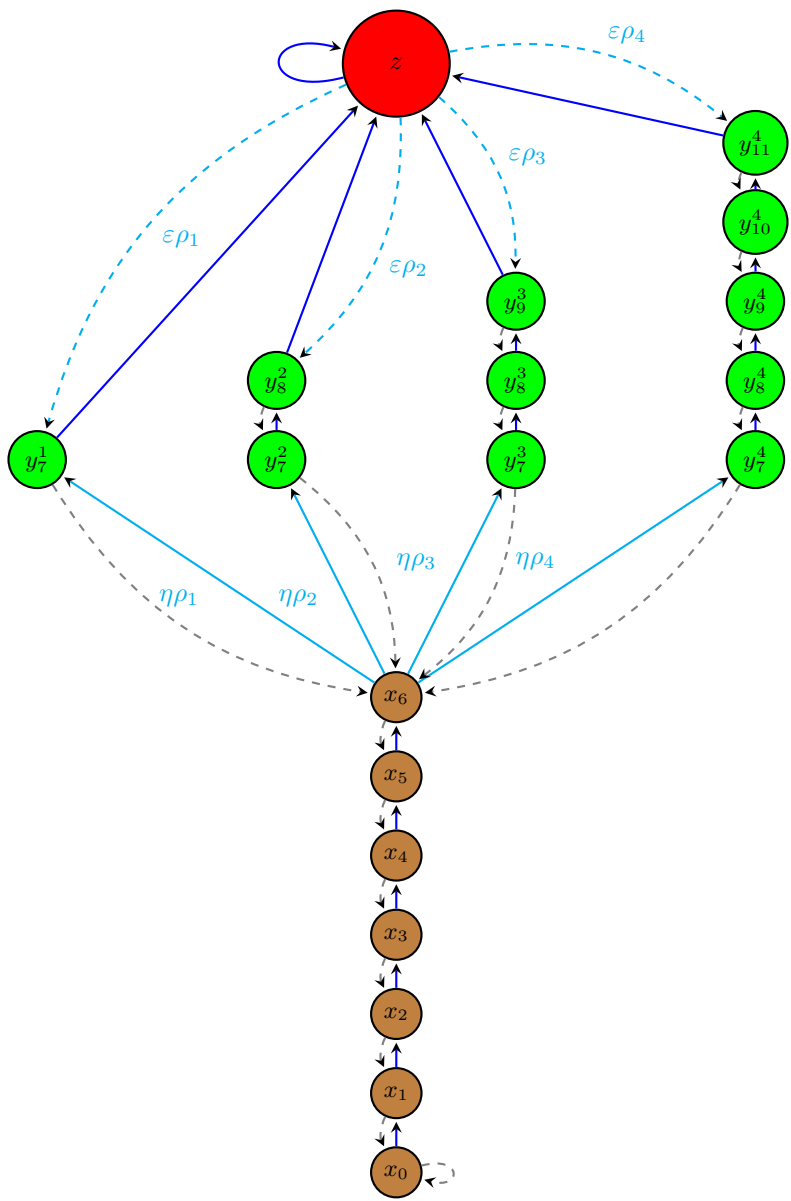

\begin{remark}
Intuitively, $x_0$ corresponds to the root, the $x_j$ correspond to the trunk, for each $i$ the $y_j^i$ correspond to the $i$-th branch, and $z$ corresponds to the unique fruit. Moreover the index $j$ of the $x_j$ or $y_j^i$ corresponds to the distance to the root $x_0$.
\end{remark}

\section{Properties of FIT Markov chains}

\subsection{Hitting and staying at the fruit}
Let $(X_t)_{t\geq 0}$ be an irreducible and aperiodic (discrete time) Markov chain on a finite state space $S$. The hitting time of a state $x\in S$ is the first time the chain reaches $x$: 
\begin{equation}
    T_x = \min \ag t\geq 0 \mid X_t = x \ad.
\end{equation}
The first return time to a vertex $x$ is defined similarly by $T_x^+ = \min \ag t\geq 1 \mid X_t = x \ad$.
We denote the probability and expectation of the chain conditioned on starting at a state $x\in S$ by $\bbP_x$ and $\bbE_x$, and if $\mu$ is a probability measure on $S$ we set $\bbP_\mu = \sum_{x\in S} \mu(x) \bbP_x$ and $\bbE_\mu = \sum_{x\in S} \mu(x) \bbE_x$.

\begin{lemma}\label{lem: hitting time the fruit bounded in expectation}
Let $(k, \boldsymbol{\ell},\boldsymbol{\rho}) \in \cP$. Let $\varepsilon\in (0,1/3)$. Let $(X_t)_{t\geq 0} = (X_t^{(\varepsilon)})_{t\geq 0} $ be the FIT Markov chain with parameters $(k, \boldsymbol{\ell},\boldsymbol{\rho}, \varepsilon)$. There exists a constant $C_0 = C_0(\ell_0)$ such that 
\begin{equation}
    \max_{x\in S} \bbE_x \cg T_z \cd \leq C_0.
\end{equation}
\end{lemma}
\begin{proof}
    Denote the transition matrix of the chain by $P$. Set $\eta= 1-\varepsilon$ as before, and set $\ell = \ell_0 + \ell_k$. Denote the probability, starting at a state $x$, to be at $z$ after $\ell$ steps, by $P^\ell(x,z)$. Then $P^\ell(x,z) \geq \eta^\ell \geq (2/3)^\ell =: \delta(\ell)$, and therefore $\max_{x\in S} \bbP_x(T_z > \ell) \leq 1-\delta(\ell)$. It follows that no matter what the starting point is, $T_z$ is stochastically dominated by $\ell$ times a geometric random variable with success probability $\delta(\ell)$, and hence its expectation is upper bounded by a constant which depends only on $\ell$, i.e.\ only on $\ell_0$ since $\ell \leq 2\ell_0$ by definition.
\end{proof}

\begin{proposition}\label{prop: hitting time of z and concentration of the stationary measure}
Let $(k, \boldsymbol{\ell},\boldsymbol{\rho}) \in \cP$. There exists $\varepsilon_0>0$ such that the following holds. Let $\varepsilon\in (0,\varepsilon_0)$. Consider the FIT Markov chain with parameters $(k, \boldsymbol{\ell},\boldsymbol{\rho}, \varepsilon)$, and denote its stationary distribution by $\pi_\varepsilon$.
    \begin{enumerate}
    \item We have $\bbE_z \cg T_z^+ \cd \leq 1+2\varepsilon$.
    \item We have $\pi_\varepsilon(z) \geq 1- 2\varepsilon$.
    \item For any fixed $t\geq 0$, we have $\bbP_z(X_t = z) \geq 1-4\varepsilon$.
    \item For any state $v$, and any times $t_1\leq t_2$, we have $\bbP_v(X_{t_2} = z) \geq \bbP_v(T_z \leq t_1) - 4\varepsilon$.
    \end{enumerate}
\end{proposition}
\begin{proof}
    \begin{enumerate}
        \item We have $\bbP_z(T_z^+ = 1) = \eta=1-\varepsilon$ and $\bbP_z(T_z^+ = 2) = \varepsilon\eta \leq \varepsilon$. Moreover, by Lemma \ref{lem: hitting time the fruit bounded in expectation}, we have $\max_{x\in S} \bbE_x \cg T_z \cd \leq C_0$ for some constant $C_0$ depending only on $\ell_0$. By the law of total expectation and since $\bbP_z(T_z^+ \geq 3) = \varepsilon^2$, we therefore get
        \begin{equation}
        \begin{split}
             \bbE_z\cg T_z^+ \cd 
            & = \bbP_z(T_z^+ = 1) + 2\bbP_z(T_z^+ = 2) + \bbE_z \cg T_z^+ \mid T_z^+ \geq 3\cd \bbP_z(T_z^+ \geq 3) \\
            & \leq (1-\varepsilon) + 2\varepsilon + \varepsilon^2(2 +\max_{x\in S} \bbE_x \cg T_z \cd) \\
            & \leq 1 + \varepsilon(1+ \varepsilon(2+C_0)),
        \end{split}
        \end{equation}
        which is less than $1+2\varepsilon$ for $\varepsilon$ small enough.
        \item By a classical property of Markov chains (\cite[Proposition 1.19]{LivreLevinPeres2019MarkovChainsAndMixingTimesSecondEdition}), we have $\pi_\varepsilon(z) = 1/\bbE_z\cg T_z^+ \cd$. Moreover by (a), $\bbE_z\cg T_z^+ \cd \leq 1+ 2\varepsilon \leq 1/(1-2\varepsilon)$. We conclude that $\pi_\varepsilon(z) \geq 1-2\varepsilon$.
\item By (b) we have $\dtv(\delta_z, \pi_\varepsilon) \leq 2\varepsilon$, where $\delta_z$ denotes the Dirac measure at $z$. Therefore
\begin{equation}
    \bbP_z(X_t = z) \geq \bbP_{\pi_\varepsilon}(X_t = z) - 2\varepsilon = \pi_{\varepsilon}(z) - 2\varepsilon \geq 1- 4\varepsilon.
\end{equation}
\item This follows immediately from (c) and the strong Markov property.
    \end{enumerate}
\end{proof}
\begin{remark}
    The bound $\pi_\varepsilon(z) \geq 1- 2\varepsilon$ in Proposition \ref{prop: hitting time of z and concentration of the stationary measure} (b) shows that if $\varepsilon$ is small then the stationary measure $\pi_\varepsilon$ is concentrated at $z$.
\end{remark}

\subsection{Hitting the fruit at specific times}

We now estimate the hitting time of the fruit $z$ starting from any state.

\begin{lemma}\label{lem: probability to reach the fruit at a given time}
    Let $(k, \boldsymbol{\ell},\boldsymbol{\rho}) \in \cP$. There exist $C_1>0$ and $\varepsilon_1>0$ such that the following holds. Let $\varepsilon\in (0,\varepsilon_1)$. Let $(X_t)_{t\geq 0} = (X_t^{(\varepsilon)})_{t\geq 0} $ be the FIT Markov chain with parameters $(k, \boldsymbol{\ell},\boldsymbol{\rho}, \varepsilon)$. Denote its stationary distribution by $\pi_\varepsilon$. 
    \begin{enumerate}
        \item Let $1\leq i\leq k$ and $\ell_0+1 \leq j \leq \ell_0 + \ell_i-1$. We have
\begin{equation}
    \bbP_{y_j^i} (T_z = \ell_0 + \ell_i - j) \geq 1- C_1\varepsilon.
\end{equation}
In particular, $\bbP_{y_j^i} (T_z \geq \ell_0) \leq C_1\varepsilon$.
        \item Let $1\leq i\leq k$ and $0\leq j \leq \ell_0$. We have
\begin{equation}
    \rho_i(1-C_1\varepsilon) \leq \bbP_{x_j} (T_z = \ell_0 - j +\ell_i) \leq \rho_i(1+C_1\varepsilon).
\end{equation}
    \end{enumerate}
\end{lemma}
\begin{proof}
    Recall that by definition $\ell_1<\ldots<\ell_k$ and $\rho_1+\ldots+\rho_k = 1$. 
    \begin{enumerate}
        \item First there is no path from $y_j^i$ to $z$ of length $\ell_0 + \ell_i - j-1$ or less, and a single one of length $\ell_0 + \ell_i - j$. To be explicit we have 
    \begin{equation}
        \bbP_{y_j^i} (T_z = \ell_0 + \ell_i - j) = \bbP(X_0 = y_j^i, X_1 = y_{j+1}^i,\ldots, X_{\ell_0 + \ell_i - j - 1} = y_{\ell_0 + \ell_i-1}^i, X_{\ell_0 + \ell_i - j} = z).
    \end{equation}
Therefore by a union bound (for the probability of the complement) we have 
\begin{equation}
    \bbP_{y_j^i} (T_z = \ell_0 + \ell_i - j) \geq 1 - (\ell_0 +\ell_i-j)\varepsilon \geq 1 - \ell \varepsilon,
\end{equation}
where $\ell = \ell_0 + \ell_k \leq 2\ell_0$.
\item Proceeding as in (a), the probability to always go up in the tree until reaching $z$, and once at $x_{\ell_0}$ to jump next to $y_{\ell_0 + 1}^i$ (or directly to $z$ if $\ell_i = 1$) is at least $\rho_i(1-\ell \varepsilon)$. The upper bound then follows from the lower bound (that we use for all $i'\ne i$): since by definition $\ell_1,\ldots,\ell_k$ are distinct integers and $\rho_1 + \ldots + \rho_k = 1$, we have
\begin{equation}
\begin{split}
      \bbP_{x_j} (T_z = \ell_0 + \ell_i - j)  & \leq 1 - \sum_{1\leq i'\leq k \du i'\ne i} \bbP_{x_j} (T_z = \ell_0 + \ell_{i'} - j) \\
      & \leq (1 -\ell\varepsilon + \ell \varepsilon) - \sum_{1\leq i'\leq k \du i'\ne i} \rho_{i'}(1-\ell\varepsilon) \\
      & = \rho_i(1-\ell\varepsilon) + \ell\varepsilon \\
      & \leq \rho_i + \ell\varepsilon \\
      & \leq \rho_i\pg 1 + (\min_{1\leq i\leq k}\rho_i)^{-1}\ell \varepsilon\pd. 
\end{split}
\end{equation}
\end{enumerate}
\end{proof}

\subsection{Estimates on the distance to stationarity}

Let $(k, \boldsymbol{\ell},\boldsymbol{\rho}) \in \cP$.  For $1\leq i\leq k$ set $t_i = \ell_0 + \ell_i$. Note that since $\ell_1<\ldots<\ell_k$, the $t_1, \ldots, t_k$ are all distinct. We define auxiliary functions $g = g_{(k, \boldsymbol{\ell},\boldsymbol{\rho})}$ from $\bbN = \ag 0, 1, \ldots \ad$ to $[0,1)$,  and $G = G_{(k, \boldsymbol{\ell},\boldsymbol{\rho})}$ from $\bbR_+$ to $[0,1]$, as follows. We set $g(t_i) = \rho_i$ for $1\leq i\leq k$, and $g(t) = 0$ for $t\in \bbN \backslash \ag t_1,\ldots,t_k\ad$; and for $x\in \bbR_+$,
\begin{equation}
    G(x) = \sum_{t = 0}^{\lf x \rf} g(t).
\end{equation}
We also set $F = F_{(k, \boldsymbol{\ell},\boldsymbol{\rho})} = 1 - G_{(k, \boldsymbol{\ell},\boldsymbol{\rho})}$.
Intuitively, the function $g$ corresponds to the hitting distribution of the fruit $z$ starting at the root $x_0$, if we took $\varepsilon = 0$ (that is if the walk had a fully deterministic behaviour), and $G$ is the associated cumulative distribution function (for the probability to be at $z$). Moreover, by definition, $G(x) = 0$ for $1\leq x \leq \ell_0$.

Let also $\varepsilon\in (0,1/2)$. Let $(X_t^{(\varepsilon)})_{t\geq 0} $ be the FIT Markov chain with parameters $(k, \boldsymbol{\ell},\boldsymbol{\rho}, \varepsilon)$. Denote again its stationary distribution by $\pi_\varepsilon$ and its state space by $S$. For $x\in S$ and $t\geq 0$, denote the law of $X_t^{(\varepsilon)}$ conditioned on $\ag X_0^{(\varepsilon)} = x \ad$ by $ \mu_{x,t}^{(\varepsilon)}$, and set
\begin{equation}
    \mathrm{d}_x^{(\varepsilon)}(t) = \dtv \pg \mu^{(\varepsilon)}_{x,t}, \pi_\varepsilon \pd.
\end{equation}
Finally, define for $t\geq 0$ the worst case total variation mixing time by
\begin{equation}
    \mathrm{d}^{(\varepsilon)}(t) = \max_{x\in S} \mathrm{d}_x^{(\varepsilon)}(t).
\end{equation}

Let us start by estimating the probability to be at the fruit at specific times starting from specific states.

\begin{lemma}\label{lem: probability to be at the fruit at a given time}
    Let $(k, \boldsymbol{\ell},\boldsymbol{\rho}) \in \cP$. There exist $C_2>0$ and $\varepsilon_2>0$ such that the following holds. Let $\varepsilon\in (0,\varepsilon_2)$. Let $(X_t)_{t\geq 0} = (X_t^{(\varepsilon)})_{t\geq 0} $ be the FIT Markov chain with parameters $(k, \boldsymbol{\ell},\boldsymbol{\rho}, \varepsilon)$. Denote its stationary distribution by $\pi_\varepsilon$.
    \begin{enumerate}
        \item Let $1\leq i\leq k$ and $\ell_0+1 \leq j \leq \ell_0 + \ell_i-1$. Let $t\geq \ell_0$. We have
\begin{equation}
    \bbP_{y_j^i} (X_t = z) \geq 1- C_2\varepsilon.
\end{equation}
\item Let $0\leq j \leq \ell_0$. Let $t\geq 0$. We have
\begin{equation}
    \bg \bbP_{x_j} (X_t = z) - \bbP_{x_0} (X_{t+j} = z) \bd \leq C_2 \varepsilon.
\end{equation}
\item Let $t\geq \ell_0 + 1$. We have
\begin{equation}
    \bg \bbP_{x_0} (X_t = z) - G(t) \bd \leq C_2\varepsilon.
\end{equation}
    \end{enumerate}
\end{lemma}
\begin{proof}
    \begin{enumerate}
\item This follows from Lemma \ref{lem: probability to reach the fruit at a given time} (a) and Proposition \ref{prop: hitting time of z and concentration of the stationary measure} (d). 
\item By the same argument as in the proof of Lemma \ref{lem: probability to reach the fruit at a given time} (a), we have $\bbP_{x_0}(X_{j} = x_{j}) \geq 1- \ell_0 \varepsilon$. The result then follows from coupling the trajectories.
\item Recall that for $1\leq i\leq k$, $t_i = \ell_0 + \ell_i$ and $g(t_i) = \rho_i$, by definition. Then the case $j=0$ of Lemma \ref{lem: probability to reach the fruit at a given time} (b) can be rewritten as 
\begin{equation}
    \bg \bbP_{x_0}(T_z = t_i) - g(t_i)\bd \leq \rho_i C_1 \varepsilon \quad \quad \text{ for } 1\leq i \leq k.
\end{equation} 
We deduce, since by definition of $G$ we have $G(t) = \sum_{1\leq i\leq k \du t_i \leq t} g(t_i)$, that
\begin{equation}
    \bg \bbP_{x_0}(T_z \leq t) - G(t) \bd \leq C_1 \varepsilon. 
\end{equation}
Moreover, we have $ \bg \bbP_{x_0}(T_z \leq t) - \bbP_{x_0}(X_t = z) \bd \leq 4\varepsilon$ by Proposition \ref{prop: hitting time of z and concentration of the stationary measure} (d). We conclude that $\bg \bbP_{x_0}(X_t = z) - G(t) \bd \leq (4+C_1) \varepsilon$. 
\end{enumerate}
\end{proof}

The next proposition estimates the total variation distance to stationarity for FIT Markov chains.

\begin{proposition}\label{prop: total variation estimate explicit with epsilon}
    Let $(k, \boldsymbol{\ell},\boldsymbol{\rho}) \in \cP$. Recall that $F= 1-G$.  There exist $C_3 = C_3(k, \boldsymbol{\ell},\boldsymbol{\rho}) >0$ and $\varepsilon_3 = \varepsilon_3(k, \boldsymbol{\ell},\boldsymbol{\rho})>0$ such that the following holds. Let $\varepsilon\in (0,\varepsilon_3)$. Let $(X_t)_{t\geq 0} = (X_t^{(\varepsilon)})_{t\geq 0} $ be the FIT Markov chain with parameters $(k, \boldsymbol{\ell},\boldsymbol{\rho}, \varepsilon)$. Denote its stationary distribution by $\pi_\varepsilon$. For any $t\geq 0$, we have
    \begin{equation}
        \bg \mathrm{d}^{(\varepsilon)}(t) - F(t)\bd \leq C_3\varepsilon.
    \end{equation}
\end{proposition}
\begin{proof}
    Let $t\geq 0$. First assume that $t\leq \ell_0$. Since the chain needs to do at least $\ell_0 + \ell_1 >\ell_0$ steps to reach $z$ from $x_0$, we have $\bbP_{x_0}(X_t = z) = 0$. Therefore, since $\pi_\varepsilon(z) \geq 1-2\varepsilon$ by Proposition \ref{prop: hitting time of z and concentration of the stationary measure} (b), we have $\mathrm{d}^{(\varepsilon)}(t) \geq \mathrm{d}_{x_0}^{(\varepsilon)}(t) \geq \pi_\varepsilon(z)\geq 1-2\varepsilon$. Since $t\leq \ell_0$ we also have $F(t) = 1$. This proves that $\bg \mathrm{d}^{(\varepsilon)}(t) - F(t)\bd \leq 2\varepsilon$. 
    
Now assume that $t\geq \ell_0 + 1$. Let us estimate $\mathrm{d}_x^{(\varepsilon)}(t)$ for different states $x$. By Proposition \ref{prop: hitting time of z and concentration of the stationary measure} (b) and (c) we have $\mathrm{d}_z^{(\varepsilon)}(t) \leq 4\varepsilon$. Moreover for $1\leq i\leq k$ and $\ell_0+1\leq j \leq \ell_0 + \ell_i-1$,
by Lemma \ref{lem: probability to be at the fruit at a given time} (a) and Proposition \ref{prop: hitting time of z and concentration of the stationary measure} (b) we have $\mathrm{d}_{y_j^i}^{(\varepsilon)}(t) \leq (2+C_2)\varepsilon$. Let $1\leq j \leq \ell_0$. By Lemma \ref{lem: probability to be at the fruit at a given time} (b) and (c), and since the function $G$ is (weakly) increasing, we have 
\begin{equation}
    \bbP_{x_j}(X_t = z) \geq \bbP_{x_0}(X_{t+j} = z) - C_2\varepsilon \leq G(t+j) - 2C_2\varepsilon \geq G(t) -2C_2\varepsilon, 
    \end{equation}
so by Proposition \ref{prop: hitting time of z and concentration of the stationary measure} (b), $\mathrm{d}_{x_j}^{(\varepsilon)}(t) \leq F(t) + (2+2C_2)\varepsilon$.
Finally, by Lemma \ref{lem: probability to be at the fruit at a given time} (c), we have $|\bbP_{x_0}(X_{t} = z) - G(t)| \leq C_2\varepsilon$ so by Proposition \ref{prop: hitting time of z and concentration of the stationary measure} (b) we have $|\mathrm{d}_{x_0}^{(\varepsilon)}(t) - F(t)| \leq (C_2+2)\varepsilon$.
It follows that 
\begin{equation}
    \bg \mathrm{d}_{x_0}^{(\varepsilon)}(t) -  \mathrm{d}^{(\varepsilon)}(t) \bd \leq C\varepsilon,
\end{equation}
where $C$ is a constant depending only on $(k, \boldsymbol{\ell},\boldsymbol{\rho})$.
We conclude by the triangle inequality that
\begin{equation}
    \bg \mathrm{d}^{(\varepsilon)}(t) - F(t) \bd \leq \bg \mathrm{d}_{x_0}^{(\varepsilon)}(t) - F(t) \bd + \bg \mathrm{d}_{x_0}^{(\varepsilon)}(t) -  \mathrm{d}^{(\varepsilon)}(t) \bd \leq (C+2+C_2)\varepsilon. \qedhere
\end{equation}
\end{proof}

\subsection{Markov chains mixing in a bounded number of steps}

Denote by $\cA$ the set of functions from $\bbN= \ag 0, 1,\ldots\ad$ to $[0,1]$ which are weakly decreasing and such that $\ag 0,1 \ad \subsetneq f(\bbN)$. For $f \in \cA$, we set $L(f) = \max \ag t\in \bbN \mid f(t) = 1\ad$, $M(f) = \max \ag t\in \bbN \mid f(t) > 0 \ad$, and $W(f) = M(f) - L(f)$. In particular, $f(t)$ is equal to 1 for $0\leq t\leq L(f)$  and equal to 0 for $t>M(f)$, and the only part where $f$ behaves non-trivially is for $L(f) < t \leq M(f)$, that is in a \textit{window} of length $W(f)$. 
Set 
\begin{equation}
    \cB = \ag f \in \cA \mid W(f) \leq L(f)  \ad.
\end{equation}

\begin{lemma}\label{lem: existence of triplets associated to a function f}
    Let $f \in \cB$. There exists $(k, \boldsymbol{\ell},\boldsymbol{\rho}) \in \cP$ such that $f = F_{(k, \boldsymbol{\ell},\boldsymbol{\rho})}$.
\end{lemma}
\begin{proof}
    Let $1= \alpha_0 > \alpha_1>\ldots >\alpha_k = 0$ be the values taken by $f$. Since the image of $f$ contains 0 and 1, and at least one element of $(0,1)$, we have $k\geq 2$. Set $\ell_0 = L(f)$. For $1\leq i\leq k$, set
\begin{equation}
    t_i = \min\ag t\in \bbN \mid f(t) = \alpha_i \ad,
\end{equation}
and $\ell_i = t_i - \ell_0$. By definition we have $1=\ell_1<\ldots<\ell_k \leq \ell_0$. Set also for $1\leq i\leq k$, 
    \begin{equation}
        \rho_i = \alpha_{i-1} - \alpha_{i}.
    \end{equation}
We have $\sum_{1\leq i\leq k} \rho_i = \sum_{1\leq i\leq k} \alpha_{i-1}- \alpha_i = \alpha_0 - \alpha_k = 1$.
Therefore setting $\boldsymbol{\ell}= (\ell_0, \ldots, \ell_k)$ and $\boldsymbol{\rho} = (\rho_1, \ldots, \rho_k)$, we have proved that $(k,\boldsymbol{\ell}, \boldsymbol{\rho})\in \cP$, and by construction we have $F_{(k,\boldsymbol{\ell}, \boldsymbol{\rho})} = f$.
\end{proof}

\begin{proposition}\label{prop: strongest result for discrete and bounded time}
    Let $f \in \cB$. There exists a sequence of (discrete time) Markov chains $((X_t)_{t\in \bbN})_{n\geq 1}$ such that 
    \begin{equation}
        \sup_{t\in \bbN} \bg \mathrm{d}^{(n)}(t) -f(t) \bd \xrightarrow[n\to\infty]{} 0,
    \end{equation}
    where for $n\geq 1$, $\mathrm{d}^{(n)}(\cdot)$ is the distance to stationarity for the chain $(X^{(n)}_t)_{t\geq 0}$.
\end{proposition}
\begin{proof}
By Lemma \ref{lem: existence of triplets associated to a function f} there exists $(k, \boldsymbol{\ell},\boldsymbol{\rho}) \in \cP$ such that $f = F_{(k, \boldsymbol{\ell},\boldsymbol{\rho})}$. For $n \geq 1$, set $\varepsilon_n = e^{-n}$, and let $(X^{(n)}_t)_{t\in \bbN}$ be the FIT the Markov chain with parameters. By Proposition \ref{prop: total variation estimate explicit with epsilon} we have
\begin{equation}
        \sup_{t\in\bbN}\bg \mathrm{d}^{(n)}(t) - f(t)\bd \leq C_3\varepsilon_n,
    \end{equation}
which tends to 0 as $n\to \infty$. This concludes the proof.
\end{proof}
\begin{remark}
A general constraint for the distance to stationarity $\mathrm{d}(t)$ of (finite irreducible aperiodic) Markov chains is its submultiplicativity “$\mathrm{d}(t) \leq 2 \mathrm{d}(t)^2$” (see \cite[Section 4.7]{LivreLevinPeres2019MarkovChainsAndMixingTimesSecondEdition}). This implies that if we have $\mathrm{d}(t) < 1/2$ for some time $t\geq 1$, then $\mathrm{d}(\cdot)$ must decrease at least exponentially fast: there exists a constant $c>0$ such that for any $s\geq 2t$ we have $\mathrm{d}(s) \leq e^{-cs/t}$. This in turn implies that it is impossible for $\mathrm{d}(\cdot)$ to have long low plateaux. In particular, the result of Proposition \ref{prop: strongest result for discrete and bounded time} would not hold anymore if we removed the assumption “$W(f) \leq L(f)$”.
\end{remark}

\subsection{Every discrete and continuous time profile is possible}
We start by proving that every discrete time profile is possible. Note that we do not assume the limiting function $p$ to be continuous, nor the \textit{window} $w_n$ or the cutoff time $t_n$ to diverge.
We recall and prove Theorem \ref{thm: discrete time profiles intro}.

\begin{theorem}
    Let $(t_n)$ be a sequence of positive integers and $(w_n)$ be a sequence of positive real numbers such that $w_n = o(t_n)$ as $n\to \infty$. Let $p\du \bbR \to [0,1]$ be a weakly decreasing function (but not necessarily continuous) such that $p(x) \xrightarrow[x\to -\infty]{} 1$ and  $p(x) \xrightarrow[x\to +\infty]{} 0$. There exists a sequence of discrete time Markov chains $((X_t^{(n)})_{t\geq 0})_{n\geq 1}$ such that
\begin{equation}
    \sup_{c\in \bbR \du t_n + cw_n \in \bbN} \bg \mathrm{d}^{(n)}(t_n + c w_n) - p(c) \bd \xrightarrow[n\to \infty]{} 0.
\end{equation}
\end{theorem}
\begin{proof}
For $n\geq 1$, set $v_n = w_n (t_n/w_n)^{1/8}$. In particular as $n\to \infty$ we have $w_n \lll v_n \lll t_n$.
Therefore there exists $n_0$ such that for $n\geq n_0$ we have $v_n \leq t_n/3$. For each $n\geq n_0$, we define a function $f_n \du \bbN \to [0,1]$ as follows:
\begin{equation}
    f_n(t) = \begin{cases}
        1 & \text{ for } 0\leq t \leq t_n - v_n, \\
        p((t - t_n)/w_n) & \text{ for } t_n - v_n < t < t_n + v_n, \\
        0 & \text{ for }  t\geq t_n + v_n.
    \end{cases}
\end{equation}
For $n\geq n_0$, $f_n \in \cA$ and since $L(f_n) \geq t_n - v_n \geq 2v_n \geq W(f_n)$, we have $f_n \in \cB$. It follows from Proposition \ref{prop: strongest result for discrete and bounded time} that for each $n\geq n_0$, there exists a Markov chain $(X^{(n)}(t))_{t\geq 0}$ such that
\begin{equation}
        \sup_{t\in\bbN}\bg \mathrm{d}^{(n)}(t) - f^{(n)}(t)\bd \leq 2^{-n}.
    \end{equation}
Moreover, by definition of $f_n$ and monotonicity for $n\geq n_0$ we have
\begin{equation}
\begin{split}
     \sup_{t\in \bbN} \bg f_n(t) - p((t-t_n)/w_n) \bd & =  \sup_{t\in \bbN \backslash (t_n - v_n, t_n+v_n)} \bg f_n(t) - p((t-t_n)/w_n) \bd \\
     & \leq |1-p(-v_n/w_n)| + |p(v_n/w_n)-0| \\
     & = 1 - p(-v_n/w_n) + p(v_n/w_n). 
\end{split}
\end{equation}
Finally, by the triangle inequality and the change of variables “$t = t_n + cw_n$”, for $n\geq n_0$ we have
\begin{equation}
\begin{split}
      & \sup_{c\in \bbR \du t_n + cw_n \in \bbN} \bg \mathrm{d}^{(n)}(t_n + c w_n) - p(c) \bd \\ \leq \, & \sup_{c\in \bbR \du t_n + cw_n \in \bbN} \bg \mathrm{d}^{(n)}(t_n + c w_n) - f_n(t_n +cw_n) \bd + \sup_{c\in \bbR \du t_n + cw_n \in \bbN} \bg f_n(t_n +cw_n) - p(c) \bd \\
        = \, & \sup_{t\in \bbN} \bg \mathrm{d}^{(n)}(t) - f_n(t) \bd + \sup_{t\in \bbN} \bg f_n(t) - p((t-t_n)/w_n) \bd \\
        \leq \, & 2^{-n} + 1 - p(-v_n/w_n) + p(v_n/w_n),
\end{split}
\end{equation}
which tends to 0 as $n\to \infty$. This concludes the proof.
\end{proof}

We now recall and prove Theorem \ref{thm: main theorem in simple continuous time form}.

\begin{theorem}
For every continuous time scaling triplet $((t_n), (w_n),p)$, there exists a sequence of continuous time Markov chains $((X^{(n)}_t)_{t\in \bbR_+})_{n\geq 1}$ such that for every $c\in \bbR$, we have 
\begin{equation*}
\mathrm{d}^{(n)}\pg t_n + c w_n\pd \xrightarrow[n\to+\infty]{} p(c).
\end{equation*}    
\end{theorem}
\begin{proof}
    Let $((t_n), (w_n),p)$ be a continuous time scaling triplet. By definition, $p$ is weakly decreasing, $p(x) \xrightarrow[x\to -\infty]{} 1$,  $p(x) \xrightarrow[x\to +\infty]{} 0$, and $w_n \lll t_n$. By Theorem \ref{thm: discrete time profiles intro} (applied to the same $p$ and $(w_n)$, but to $(\lf t_n \rf)$), there exists a sequence of discrete time Markov chains $((X_t^{(n, \#)})_{t\in \bbN})_{n\geq 1}$ such that as $n\to \infty$, 
    \begin{equation}
    \sup_{c\in \bbR \du \lf t_n \rf  + cw_n \in \bbN} \bg \mathrm{d}^{(n, \#)}(\lf t_n \rf + c w_n) - p(c) \bd \xrightarrow[n\to \infty]{} 0.
\end{equation}
Set $u_n = \sqrt{w_n \sqrt{t_n}}$ for $n\geq 1$. Since $w_n \ggg \sqrt{t_n}$ by definition of a continuous time scaling triplet, we have $\sqrt{t_n} \lll u_n \lll w_n$. Let $(N_t)_{t\geq 0}$ be a Poisson process with parameter 1. For each $n\geq 1$, let $(X^{(n)}_t)_{t\in \bbR_+} = (X_{N_t}^{(n, \#)})_{t\in \bbN}$ be the rate 1 continuous time analog of $(X_t^{(n, \#)})_{t\in \bbN}$, with jump times given by $(N_t)_{t\geq 0}$.
Let $c\in \bbR$. For $n\geq 1$ consider the event
 \begin{equation}
     E_{n,c}= \ag \lf t_n + c w_n - u_n \rf \leq N_{t_n + c w_n} \leq  \lc t_n + c w_n + u_n \rc \ad.
 \end{equation}
Since $u_n \ggg \sqrt{t_n} \sim \sqrt{t_n + c w_n}$ we have $\bbP(E_{n,c}) = 1-o(1)$.
Moreover, as $n\to \infty$, by continuity of $p$ we have $\mathrm{d}^{(n, \#)}(t) \to p(c)$ uniformly over all integers $t\in[ \lf t_n + c w_n - u_n \rf, \lc t_n + c w_n + u_n \rc]$.
We conclude that
\begin{equation}
    \mathrm{d}^{(n)}(t_n + c w_n) \to p(c).
    \qedhere
\end{equation}
\end{proof}

\section{Pathological examples}

In all Markov chains for which there is a sharp phase transition and the behaviour within the phase transition is well understood, there is a unique order of magnitude to zoom in around the cutoff time, and the fluctuations observed are asymptotically continuous. However, in general, even when a sequence of Markov chains exhibits a cutoff, what happens when we further zoom in, in the phase transition, can be arbitrarily complicated. Also, the distance to stationarity can decrease peculiarly even for some conjugacy invariant random walks on groups.

\subsection{Different behaviours in uncountably many windows}\label{s: A non-trivial behaviour in uncountably many windows}

Let $n\geq 8$. Set $\ell_0 = n$ and $k = k^{(n)} = \lf \ln \ln n \rf$. Set $\ell_i = \lf n^{i/(k+1)} \rf$ and $\rho_i = \rho_i^{(n)} = 1/k$ for $1\leq i\leq k$. Set $\varepsilon = 2^{-n}$. Denote by $(W^{(n)}_t)_{t\geq 0}$ the FIT Markov chain with parameters $(k, \boldsymbol{\ell}, \boldsymbol{\rho}, \varepsilon)$. As previously, denote $\mathrm{d}^{(n)}(\cdot)$ the associated distance to stationarity. This defines a sequence of Markov chains $((W^{(n)}_t)_{t\geq 0})_{n\geq 8}$. 

\begin{proposition}
    Consider the sequence of FIT Markov chains $((W^{(n)}_t)_{t\geq 0})_{n\geq 8}$ defined above. For every $\alpha\in (0,1)$ and every $c\in \bbR$, as $n\to \infty$ we have 
    \begin{equation}
        \mathrm{d}^{(n)}(\lf n + c n^{\alpha} \rf) \to \begin{cases}
            1 & \text{ if } c\leq 0 \\
            1-\alpha & \text{ if } c>0. 
        \end{cases}
    \end{equation}
\end{proposition}
\begin{proof}
     Let $\alpha\in (0,1)$ and $c\in \bbR$. Starting from the root it is impossible to reach the fruit in $n$ steps or less, therefore if $c\leq 0$, we have $\bbP_{x_0^{(n)}}(W^{(n)}_{\lf n - cn^{\alpha}\rf} = z^{(n)}) = 0$, and therefore $\mathrm{d}^{(n)}(\lf n + c n^{\alpha} \rf) = 1-o(1)$. 
     
     Assume now that $c>0$. By construction, the probability to always go up before time $\lf n + c n^{\alpha} \rf$ and to take one of the $(\alpha+o(1))k^{(n)}$ branches of length at most $\lf c n^{\alpha}\rf$ is $\sum_{i\leq (\alpha+o(1))k^{(n)}} \rho_i^{(n)} = \alpha + o(1)$. Therefore we have $\bbP_{x_0^{(n)}}(W^{(n)}_t = z^{(n)}) = \alpha + o(1)$, and it follows that $\mathrm{d}^{(n)}(\lf n + c n^{\alpha} \rf) = 1-\alpha + o(1)$. 
\end{proof}

In the previous example, for each $\alpha\in (0,1)$, the distance to stationarity at time $n + cn^\alpha$ is asymptotically a piecewise constant function. In other words, for each $\alpha\in (0,1)$, there is a non-trivial limiting function when we zoom in in a window of size $\Theta(n^\alpha)$. 

\subsection{Pseudo-laziness for random walks on Cayley graphs}

All states of FIT Markov chains have a specific and distinct role. One may wonder if complicated mixing time behaviour can arise from very structured objects. For $n\geq 1$ denote respectively by $\kS_n$ and $\kA_n$ the symmetric and alternating groups of index $n$; and by $\cC_{2}^{(n)}$ and $\cC_{5}^{(n)}$ the conjugacy classes of transpositions and 5-cycles. 

\medskip

First consider the sequence of (discrete time) Markov chains whose increment measure is uniform on $\kA_n \cup \cC_{2}^{(n)}$. There the distance to stationarity is approximately $1$ at time $0$, drops to $1/2$ at time $1$ since the walk is already almost uniform on $\kA_n$, and then decreases very slowly (the time scale is $|\kA_n|/|\cC_{2}^{(n)}|\sim n!/n^2$), exponentially, to 0.

Similarly, consider consider the sequence of (discrete time) Markov chains whose increment measure is uniform on $\cC_{5}^{(n)} \cup \cC_{2}^{(n)}$. There with high probability only 5-cycles are picked for a long time. The distance to stationarity is therefore close to 1 up to time $\frac{1}{5}n\ln n$ where it abruptly drops to $1/2$ since the walk is then almost uniform on $\kA_n$, and then decreases slowly (the time scale is $|\cC_{5}^{(n)}|/|\cC_{2}^{(n)}|\asymp n^3$), exponentially, to 0.

\medskip

This illustrates that even conjugacy invariant random walks on groups can have peculiar behaviours. In the examples above, this is due to the fact that to mix one needs to uniformize also the sign of permutations. These examples are almost 2-periodic, and transpositions play the role of laziness.

\subsection{An example with countably many nested windows}\label{s: Countably many nested windows}

FIT Markov chains enable constructing very diverse behaviours. It is possible to see a non-trivial behaviour at some zooming scale $(w_n)$, with a limiting function which decreases continuously on $(-\infty,0)$ from $1$ to $3/4$ and on $(0, +\infty)$ from $1/4$ to $0$, but with a discontinuity at 0. Then we can still adjust the lengths of the branches to have essentially any limiting function when we further zoom in. As we now describe it is even possible to have countably many nested windows.

\medskip

Let $n\geq 8$. Set $\ell_0 = 2n$ and $k = 2^{L+1}$, where $L = \lf \ln \ln n \rf$. For $1\leq i\leq k$, set $\rho_i = 1/k$. For $2\leq q\leq L$, and $L2^{1-q} \leq i\leq L2^{2-q}$ set $m_i = \lc n^{1/q}2^{1-q}\rc$. 
For $1\leq i\leq k/2$, set 
\begin{equation}
    \ell_{k/2 +i} = n + \sum_{j=1}^im_j \quad\quad \text{ and } \quad\quad \ell_{k/2 +1- i} = n - \sum_{j=1}^im_j.
\end{equation}
Set $\varepsilon = 2^{-n}$. Denote by $(Y^{(n)}_t)_{t\geq 0}$ the FIT Markov chain with parameters $(k, \boldsymbol{\ell}, \boldsymbol{\rho}, \varepsilon)$. This defines a sequence of Markov chains $((Y^{(n)}_t)_{t\geq 0})_{n\geq 8}$. This construction leads immediately to the following proposition, whose proof we omit.

\begin{proposition}
    Consider the sequence of FIT Markov chains $((Y^{(n)}_t)_{t\geq 0})_{n\geq 8}$ defined above. For every integer $q\geq 2$ and every $c\in \bbR$, as $n\to \infty$ we have 
    \begin{equation}
        2^{q}\pg  \mathrm{d}^{(n)}\pg \lf 3n + c n^{1/q} \rf\pd - \frac{1}{2}\pd \to \begin{cases}
            2 & \text{ if } c\leq -1, \\
            (1-x) & \text{ if } -1<c<0, \\ 
            0 & \text{ if } c = 0, \\
            (-1-x) & \text{ if } 0<c<1, \\
            -2 & \text{ if } c\geq 1. \\
        \end{cases}
    \end{equation}
\end{proposition}

\subsection{One chain to rule them all}

Denote by $\cD$ the set of continuous weakly decreasing functions $p\du \bbR \to [0,1]$, such that $\lim_{x\to -\infty} p(x) = 1$ and $\lim_{x\to +\infty} p(x) = 0$.

Let $k\geq 1$. Let $\cE_k$ be the set of increasing $(k+1)$-tuples of rational numbers $\boldsymbol{q} = (q_0, q_1, \ldots, q_k)$. For each $\boldsymbol{q} \in \cE_k$, let $a\du \bbR \to [0,1]$ defined by
\begin{equation}
    a(x) = \begin{cases}
            1 & \text{ if } x < q_0, \\
            0 & \text{ if } x\geq q_k, \\
            1-\frac{i}{k} - \frac{1}{k} \frac{x - q_i}{q_{i+1}- q_i} & \text{ if } q_i\leq x <q_{i+1} \text{ for some } 1\leq i\leq k-1.
        \end{cases}
\end{equation}
In particular, $a$ is piecewise affine, and $a(q_i) = 1-i/k$ for $1\leq i\leq k$. Denote the set of all such functions by $\cL_k$. In particular, since $\cE_k$ is countable, so is $\cL_k$. Therefore, $\cL := \cup_{k\geq 1} \cL_k$ is a countable dense subset of $\cD$, with respect to the distance defined for $f,g\in \cD$ by
\begin{equation}
    d_\infty(f, g) = \sup_{x\in \bbR} |f(x) - g(x)|.
\end{equation}
Since $\cL$ is countable, we may write  $\cL = \ag a_m,  m\in \bbN \ad$.
Let $(t_n)_{n\geq 1}$ be a sequence of positive integers, and $(w_n)_{n\geq 1}$ be a sequence of positive real numbers, such that $w_n = o(t_n)$. 
By Theorem \ref{thm: discrete time profiles intro}, for each $(m,n)\in (\bbN^*)^2$ there exists a Markov chain $(X^{(m, n)}_t)_{t\geq 0}$ such that 
\begin{equation}
    \sup_{c\in \bbR \du t_n + cw_n \in \bbN} \bg \mathrm{d}^{(m, n)}(t_n + c w_n) - a_m(c) \bd \leq 2^{-(m+n)}.
\end{equation}
Interleaving $(\bbN^*)^2$ into $\bbN^ *$ we immediately obtain that all cutoff profiles can be realized as subsequential limits coming from a single sequence of Markov  chains. 
\begin{theorem}  
    Let $(t_n)_{n\geq 1}$ be a sequence of positive integers, and $(w_n)_{n\geq 1}$ be a sequence of positive real numbers, such that $w_n = o(t_n)$. There exists a sequence of Markov chains $(X_t^{(n)})_{t\geq  0})_{n\geq 1}$ such that for every $p\in \cD$, there exists $\varphi \du \bbN^* \to \bbN^*$ strictly increasing such that 
\begin{equation}
     \sup_{c\in \bbR \du t_{\varphi(n)} + cw_{\varphi(n)} \in \bbN} \bg \mathrm{d}^{({\varphi(n)})}(t_{\varphi(n)} + c w_{\varphi(n)}) - p(c) \bd \xrightarrow[n\to\infty]{} 0.
\end{equation}
\end{theorem}

\section{Comments and open questions}
This section contains some comments and questions about profiles for diffusions on compact Lie groups, how to obtain profiles without understanding the eigenvectors, the separation profile for random transpositions, and the constraints on cutoff profiles for vertex-transitive graphs.

\medskip

Finding cutoff profiles is often done via spectral techniques. This requires a detailed understanding not only of the eigenvalues of a Markov chain, but also of the interactions between its eigenvectors.
Méliot \cite{Méliot2014} found explicit formulas for the densities of the Brownian motion on a number of compact symmetric spaces -- including the classical simple compact Lie groups -- and proved in Theorem 6 that all of them exhibit a cutoff phenomenon. Some computations for special unitary groups point towards compensations between the eigenvectors that involve the trace of matrices. (We note that the number of fixed points of a permutation is the trace of the associated permutation matrix.) This gives hope for finding the cutoff profiles for these diffusions.
\begin{question}
    Can we find the profiles for the Brownian motions on compact symmetric spaces?
\end{question}

For the random-to-random shuffle, the eigenvalues were found in \cite{DiekerSaliola2018randomtorandom} and the cutoff was proved in \cite{BernsteinNestoridi2019}. However, the eigenvectors have complicated recursive expressions that are difficult to manipulate, and finding the limiting profile is still open, see the open question in \cite[Section 6.2]{Nestoridi2024comparisonstar}.
Recently, Jain and Sawhney \cite{JainSawhney2024transpositionprofileotherproof} found another way to obtain the profile for random transpositions. Instead of understanding the compensations between eigenvectors, as done in \cite{Teyssier2020} (characters are sums of eigenvectors), they extract the limiting observables (the fixed points) and the associated eigenvalues using probabilistic arguments and discard the other observables via an $L^2$ bound.
These probabilistic arguments are delicate but offer an alternative way to find profiles that doesn’t rely on understanding the eigenvectors. It might be suited for chains such as the random-to-random shuffle where the eigenvectors are cumbersome to understand.

\medskip

One may also want to understand cutoff profiles for distances other than the total variation distance. Computing the $L^2$ distance only requires knowledge of the eigenvalues. In some cases obtaining information on all eigenvalues is difficult, but upon reaching such information the $L^2$ profile is relatively easy to compute. This is the case (in discrete time) for random transpositions and also for the random-to-random shuffle.

Another distance of particular interest is the separation distance, defined as follows in the case of transitive graphs. Let $\Gamma = (V, E)$ be a finite connected vertex-transitive graph and let $(X_t)_{t\geq 0}$ be the simple random walk on $\Gamma$, in discrete or continuous time, and started at some vertex $o\in V$. For $t\geq 0$, and $x,y\in V$, denote $p_t(x,y) = \bbP_x(X_t = y)$. For $t\geq 0$ the minimum of likelihood at time $t$ is the quantity $\mlh(t):= \min_{x\in V}p_t(o,x)$, and the separation distance is defined by
\begin{equation}
    \dsep(t) := 1 - |V| \mlh(t).
\end{equation}

\begin{question}\label{q: minimum of likelihood transpositions}
    For random transpositions on $\kS_n$, in continuous time and started at the identity permutation, is it true that the minimum of likelihood is attained at $n$-cycles for every $t\geq 0$?
\end{question}

Let us sketch how to compute the likelihood of $n$-cycles. For $n\geq 1$, denote by $\Id_n$ the identity permutation, and let $\gamma_n = (1 \, \ldots \, n)$, which is an $n$-cycle. Recall that irreducible representations $\lambda\in \widehat{\kS_n}$ of $\kS_n$ are indexed by Young diagrams, and that they index the eigenvalues $e_\lambda$ of the chain. It also follows from the Murnaghan--Nakayama rule that characters evaluated at an $n$-cycle are equal to $\pm 1$ if $\lambda$ is hook-shaped, and $0$ otherwise. Therefore, by Fourier inversion on $\kS_n$, omitting some technical details and integer parts, we have at $t_{n,c} = \frac{1}{2}n(\ln n + c))$
\begin{equation}
\begin{split}
    n!p_{t_{n,c}}(\Id_n, \gamma_n) = \sum_{\lambda \in \widehat{\kS_n}} d_\lambda e_\lambda^{t_{n,c}} \ch^\lambda(\gamma_n) & = o(1) + \sum_{j=0}^{n/2} \binom{n}{j} n^{-j}e^{-jc} (-1)^j \\
    & = o(1) + \sum_{j\geq 0}\frac{(-e^{-c})^j}{j!} = o(1) +  e^{-e^{-c}}.
\end{split} 
\end{equation}

A positive answer to Question \ref{q: minimum of likelihood transpositions} would therefore imply the following conjecture.

\begin{conjecture}
    Consider, for $n\geq 2$, the random transposition shuffle $(X_t^{(n)})_{t\geq 0}$ on $\kS_n$, in discrete or continuous time, and denote the associated separation distance to stationarity at time $t\geq 0$ by $\mathrm{d}_{\textup{sep}}^{(n)}(t)$. For any $c\in \bbR$ we have
    \begin{equation}
        \mathrm{d}_{\textup{sep}}^{(n)} \pg \lf \frac{1}{2}n(\ln n + c) \rf\pd \xrightarrow[n\to \infty]{} 1-e^{-e^{-c}}.
    \end{equation}
In words, the separation profile for random transpositions is given by a Gumbel distribution.
\end{conjecture}

Finally, the states of the fruit-inosculated-tree Markov chains defined in this paper play very different roles, and the stationary measure is concentrated at one state. 
We informally ask the following. 
\begin{question}
    How is the statement of Theorem \ref{thm: main theorem in simple continuous time form} affected if we add restrictions on the sequence of Markov chains?
\end{question}
We believe, for instance, that restricting to simple random walks on regular graphs is a weak constraint, but that not every cutoff profile can be realized by simple random walks on vertex-transitive graphs, even without a restriction on the degree.
\begin{question}
    Does the statement of Theorem \ref{thm: main theorem in simple continuous time form} still hold for simple random walks on vertex-transitive graphs?
\end{question}

\section*{Acknowledgements}

We thank Persi Diaconis, Jonathan Hermon, Peter Nejjar, Evita Nestoridi, Justin Salez, and Dominik Schmid for useful conversations and comments. Part of the research was conducted as the author was supported by the Pacific Institute for the Mathematical Sciences, the Centre national de la recherche scientifique, and the Simons foundation, via a PIMS-CNRS-Simons postdoctoral fellowship.

\bibliographystyle{alpha}
\bibliography{bibliographieLucas}
\end{document}